\documentclass{amsart}

\usepackage{amssymb}
\usepackage{amsmath}
\usepackage{amsthm}
\usepackage[pdftex]{graphicx} 
\usepackage{enumerate}

\newtheorem{theorem}{Theorem}[section]
\newtheorem{lemma}[theorem]{Lemma}

\newtheorem{proposition}[theorem]{Proposition}

\theoremstyle{definition}
\newtheorem{example}[theorem]{Example}
\newtheorem{remark}[theorem]{Remark}
\newtheorem{definition}[theorem]{Definition}

\newtheorem{questions}[theorem]{Questions}

\numberwithin{equation}{theorem}

\newcommand{\pa}{\theta_1}
\newcommand{\pb}{\theta_2}
\newcommand{\pc}{\theta_3}

\begin{document}
\title{The $q$-Division Ring and its Fixed Rings}
\author{Si\^an Fryer}
\address{School of Mathematics, The University of Manchester, Manchester M13 9PL, UK}
\email{sian.fryer@manchester.ac.uk}
\begin{abstract}
We prove that the fixed ring of the $q$-division ring $k_q(x,y)$ under any finite group of monomial automorphisms is isomorphic to $k_q(x,y)$ for the same $q$.  In a similar manner, we also show that this phenomenon extends to an automorphism that is defined only on $k_q(x,y)$ and does not restrict to $k_q[x^{\pm1},y^{\pm1}]$.  We then use these results to answer several questions posed by Artamonov and Cohn about the endomorphisms and automorphisms of $k_q(x,y)$.
\end{abstract}

\maketitle

\setcounter{section}{-1}
\section{Introduction}
Let $k_q[x,y] =k\{ x,y \} / (xy-qyx)$ denote the \textit{quantum plane}, where $k$ is any field of characteristic zero and $q \in k^{\times}$.  Since $k_q[x,y]$ is a Noetherian domain its full ring of fractions is a division ring, denoted by $D$ or $k_q(x,y)$.

In \cite{Artin} Artin gave a conjectured classification of the non-commutative surfaces up to birational equivalence, a conjecture which has informed and motivated much of the research in non-commutative algebra and algebraic geometry in the intervening twenty years. For recent work relating to this, see for example \cite{jason1}, \cite{daniel1}, \cite{sue1}.

For algebraists, the interest in this conjecture lies in its rephrasing in terms of division rings.  Stated informally, the conjecture predicts that the only division rings appearing as function fields of non-commutative surfaces are: (i) division rings of algebras finite-dimensional over function fields of transcendence degree 2; (ii) division rings of Ore extensions of function fields of curves; (iii) the degree 0 part of the graded division ring of the 3-dimensional Sklyanin algebra.  Artin's conjecture is discussed in more detail in \cite{SV1}, which also includes a definition of the Sklyanin algebra.

This motivates an investigation of the fixed rings of these division rings under finite groups of automorphisms; it is a natural area where new division rings not already appearing on this list might be found.  We will focus on $k_q(x,y)$, where some results are already known: for example when the group of automorphisms restricts to automorphisms on $k_q[x,y]$.
\begin{theorem}\label{AD_prop}\cite[Proposition~3.4]{AD1} Let $k$ be a field of characteristic zero, and $q \in k^{\times}$ not a root of unity.  Denote by $R_q$ the quantum plane $k_q[x,y]$ and by $D_q$ its full ring of fractions.  Then:
\begin{enumerate}[(i)]
\item For $q' \in k$, we have $D_q \cong D_{q'}$ if and only if $q' = q^{\pm1}$, if and only if $R_q \cong R_{q'}$.
\item For all finite subgroups $G$ of $Aut(R_q)$, $D_q^G \cong D_{q'}$ for $q' = q^{|G|}$.
\end{enumerate}
\end{theorem}
The methods used in the proof of this result do not generalise to groups of automorphisms defined only on $D$, nor even to those defined on $k_q[x^{\pm1},y^{\pm1}]$ and extended to $D$.  And nor should we expect them to, as the following result demonstrates:
\begin{theorem}\cite[\S13.6]{SV1}\label{order_2_monomial_result}
Let $\tau$ be the automorphism defined on $D$ by
\[\tau: x \mapsto x^{-1}, \ y \mapsto y^{-1}\]
Then the fixed ring $D^{\tau}$ is isomorphic to $D$, for the same choice of $q$.
\end{theorem}
The map $\tau$ is an example of a \textit{monomial automorphism}: one where the images of $x$ and $y$ are both monomials (for the precise definition, see Definition~\ref{def_monomial_autos}).  A natural question to ask is whether Theorem~\ref{order_2_monomial_result} should extend to all automorphisms of this form; another is what happens when we consider automorphisms which do \textit{not} restrict to the quantum Laurent polynomial ring $k_q[x^{\pm1},y^{\pm1}]$.

In \S\ref{fixed ring} we define an automorphism of order 2 on $D$ which does not restrict to $k_q[x^{\pm1},y^{\pm1}]$ and examine its fixed ring, which we show is once again isomorphic to $D$.  In \S\ref{more fixed rings} we extend Theorem~\ref{order_2_monomial_result} to cover all finite groups of monomial automorphisms on $D$; combined, these give rise to the following theorem.
\begin{theorem}[Theorem~\ref{epic_theorem}, Theorem~\ref{thm_monomial_results_2}]\label{thm_monomial_results} Let $k$ be a field of characteristic zero and $q \in k^{\times}$.
\begin{enumerate}[(i)]
\item Define an automorphism on $D$ by
\[\varphi: x \mapsto (y^{-1}-q^{-1}y)x^{-1}\quad y \mapsto -y^{-1}\]
and let $G$ be the group generated by $\varphi$.  Then $D^G \cong D$ for the same choice of $q$.
\item Suppose $k$ contains a third root of unity $\omega$, and both a second and third root of $q$.  If $G$ is a finite group of monomial automorphisms of $D$ then $D^G \cong D$, again for the same choice of $q$.
\end{enumerate}
\end{theorem}
This is in contrast to Theorem~\ref{AD_prop}, where the choice of $q$ in the fixed ring depended on the order of the group.

Finally in \S\ref{consequences} we turn our attention to the homomorphisms of $D$, and tackle several open questions posed by Artamonov and Cohn in 1999.  In their paper \cite{AC1}, they prove that any homomorphism from $D$ to itself decomposes into a product of well-behaved ``elementary automorphisms'' and a certain conjugation map $c_z$.  Here $c_z(r) = zrz^{-1}$ is conjugation by an element $z \in k_q(y)(\!(x)\!)$, the Laurent power series ring defined in \S\ref{notation} below.

Artamonov and Cohn ask whether these conjugation maps must always be inner, i.e. whether the conjugating element they construct will always be in $D$ itself.  They also ask whether any endomorphism of $D$ must be an automorphism, a question also considered by Alev and Dumas in \cite{AD4}.  
In \S\ref{consequences} we answer both of these questions in the negative with the following theorem.
\begin{theorem}[Theorem~\ref{thm_AC_results_2}]\label{thm_AC_results}Let $k$ be a field of characteristic zero and $q \in k^{\times}$ not a root of unity.  Then:
\begin{enumerate}[(i)]
\item The $q$-division ring $D$ admits examples of bijective conjugation maps by elements $z \not\in D$; these include examples satisfying $z^n \in D$ for some positive $n$, and also those such that $z^n \not\in D$ for all $n \geq 1$.
\item $D$ also admits an endomorphism which is not an automorphism, which can be represented in the form of a conjugation map.  
\end{enumerate}
\end{theorem}
We also prove results indicating a possible direction for further study of the automorphism group $Aut(D)$, and finish by listing several new open questions raised by our results.

\textbf{Acknowledgements.} The author would like to thank their PhD supervisor, Toby Stafford, for suggesting the original problem and for innumerable helpful conversations and suggestions during the work.  This research was supported by an EPSRC Doctoral Training Award, and we gratefully acknowledge their support.

\section{Definitions and Notation}\label{notation}
We begin by outlining the notation and definitions we will need.  Throughout, fix $k$ to be a field of characteristic zero and $q \in k^{\times}$.  For most of this paper we require no conditions on $q$ except that it be non-zero, although in \S\ref{consequences} we will need to restrict to the case of $q$ not a root of unity.

Let $R$ be any ring, $\alpha$ an endomorphism of $R$ and $\delta$ a left $\alpha$-derivation.  The (left) \textit{Ore extension} $R[x;\alpha, \delta]$ is an overring of $R$, which is free as a left $R$-module with basis $\{1,x,x^2, \dots\}$ and commutation relation
\[xr = \alpha(r)x + \delta(r).\]
We write $R[x; \alpha]$ or $R[x;\delta]$ when $\delta = 0$ or $\alpha = 1$ respectively.

The ring $k_q[x,y]$ can be viewed as the Ore extension $k[y][x;\alpha]$, where $\alpha$ is the automorphism defined on $k[y]$ by $\alpha(y) = qy$.  For $r \in k_q[x,y]$, let $deg_x(r)$ be the degree of $r$ as a polynomial in $x$.

By localizing $k_q[x,y]$ at the set of all its monomials, we obtain the ring of \textit{quantum Laurent polynomials} $k_q[x^{\pm1},y^{\pm1}]$.  This ring sits strictly between $k_q[x,y]$ and the division ring $k_q(x,y)$, and the properties of it and its fixed rings are studied in \cite{Baudry}.  

The \textit{$q$-division ring} $D = k_q(x,y)$ embeds naturally into a larger division ring, namely the ring of Laurent power series
\[k_q(y)(\!(x)\!) = \left\{ \sum_{i \geq n} a_i x^i : n \in \mathbb{Z}, a_i \in k(y)\right\}\]
subject to the same relation $xy = qyx$.  It is often easier to do computations in $k_q(y)(\!(x)\!)$ than in $D$, and we will identify elements of $D$ with their image in $k_q(y)(\!(x)\!)$ without comment.

If $R$ is any ring and $z$ an invertible element, we denote the resulting conjugation map on $R$ by
\[c_z: r \mapsto zrz^{-1} \quad \forall r \in R.\]
Since we will define conjugation maps on $D$ with $z \in k_q(y)(\!(x)\!) \backslash D$, the following distinction will be important: we call a conjugation map $c_z$ an \textit{inner automorphism} on $R$ if $z, z^{-1} \in R$.

Finally, if $G$ is a subgroup of $Aut(R)$ we define the fixed ring to be
\[R^G = \{r \in R: g(r) = r, \ \forall g \in G\}.\]
If $G = \langle \varphi \rangle$ is cyclic, we will also denote the fixed ring by $R^{\varphi}$.

\section{An automorphism and its fixed ring}\label{fixed ring}
Define on $k_q(x,y)$ the map
\begin{equation}\label{def_varphi}\varphi: x \mapsto (y^{-1}-q^{-1}y)x^{-1}, \quad y \mapsto -y^{-1}.\end{equation}
Since $x$ only appears once in the image, it is easy to see that these images $q$-commute and so this is a homomorphism.  We can also easily check that it has order 2, since
\begin{align*}
\varphi^2(x) &= \varphi\Big((y^{-1}-q^{-1}y)x^{-1}\Big) \\
&= (-y+q^{-1}y^{-1})x(y^{-1}-q^{-1}y)^{-1} \\
&= (q^{-1}y^{-1}-y)(q^{-1}y^{-1}-y)^{-1}x \\
&= x
\end{align*}
and it is therefore an automorphism on $k_q(x,y)$. The aim of this section is to prove the following result.
\begin{theorem}\label{epic_theorem} Let $G$ be the group generated by $\varphi$.  Then $k_q(x,y)^G \cong k_q(x,y)$ as $k$-algebras.\end{theorem}

Before tackling the proof of this theorem, we will need some subsidiary results.

Recall that the algebra generated by two elements $u$, $v$ subject to the relation $uv-qvu = \lambda$ (for some $\lambda \in k^{\times}$) is called the \textit{quantum Weyl algebra}.  This ring also has a full ring of fractions, which can be seen to be equal to $D$ by sending $u$ to the commutator $uv-vu$ \cite[Proposition~3.2]{AD2}.

We will construct a pair of elements in $k_q(x,y)^G$ which satisfy a quantum Weyl relation and show that they generate the fixed ring.  A simple change of variables then yields the desired isomorphism.

In order to simplify the notation, set $\Lambda = y^{-1} - q^{-1}y$.  Inspired by \cite{division} and \cite[\S13.6]{SV1}, we define our generators using a few simple building blocks.  We set
\begin{equation}\begin{gathered}\label{abc_defs}
a= x - \Lambda x^{-1},\qquad  b= y+y^{-1}, \qquad c= xy + \Lambda x^{-1}y^{-1} \\
 h = b^{-1}a, \qquad g = b^{-1}c\end{gathered}
\end{equation}
and verify that $h$ and $g$ satisfy the required properties.
\begin{lemma} 
The elements $h$ and $g$ are fixed by $\varphi$ and satisfy the relation 
\[hg-qgh = 1-q.\]\end{lemma}

\begin{proof}
The first statement is trivial, since $\varphi$ acts on $a$, $b$ and $c$ as multiplication by $-1$.

After multiplying through by $b$, we see that the equality $hg - qgh = 1-q$ is equivalent to 
\[ab^{-1}c - qcb^{-1}a = (1-q)b\]
which allows us to verify it by direct computation.  Indeed,
\begin{equation}\label{abc}\begin{aligned}
 ab^{-1}c &= (x-\Lambda x^{-1})(y+y^{-1})^{-1}(xy + \Lambda x^{-1}y^{-1}) \\
&= \Big((qy + q^{-1}y^{-1})^{-1}x - \Lambda(q^{-1}y + qy^{-1})^{-1}x^{-1}\Big)\big(xy + \Lambda x^{-1}y^{-1}\big) \\
&= q^2y(qy + q^{-1}y^{-1})^{-1}x^2 + \alpha(\Lambda)(qy + q^{-1}y^{-1})^{-1}y^{-1} \\
& \quad \qquad - \Lambda(q^{-1}y + qy^{-1})^{-1}y - q^2y^{-1}\Lambda \alpha^{-1}(\Lambda)(q^{-1}y + qy^{-1})^{-1}x^{-2}
\end{aligned}\end{equation}
\begin{equation}\label{qcba}\begin{aligned}
qcb^{-1}a &= q(xy + \Lambda x^{-1}y^{-1})(y+y^{-1})^{-1}(x-\Lambda x^{-1}) \\
&= q\Big(qy(qy + q^{-1}y^{-1})^{-1}x + q\Lambda y^{-1}(q^{-1}y + qy^{-1})^{-1}x^{-1}\Big)\big(x-\Lambda x^{-1}\big) \\
&= q^2y(qy+q^{-1}y^{-1})^{-1}x^2 - q^2y\alpha(\Lambda)(qy + q^{-1}y^{-1})^{-1} \\
&\quad \qquad + q^2\Lambda y^{-1}(q^{-1}y + qy^{-1})^{-1} - q^2y^{-1}\Lambda \alpha^{-1}(\Lambda)(q^{-1}y + qy^{-1})^{-1} x^{-2}
\end{aligned}\end{equation}
Putting these together, we see that the terms in $x^2$ and $x^{-2}$ cancel out, leaving us with
\begin{align*}
 ab^{-1}c - qcb^{-1}a &= \alpha(\Lambda)(qy + q^{-1}y^{-1})^{-1}(y^{-1} + q^2y) \\
& \quad \qquad - \Lambda(q^{-1}y + qy^{-1})^{-1}(y+q^2y^{-1}) \\
&= \alpha(\Lambda)q - \Lambda q \\
&= (q^{-1}y^{-1} - y)q - (y^{-1}-q^{-1}y)q\\
&=(1-q)b \qedhere
\end{align*}\end{proof}
Let $R$ be the division ring generated by $h$ and $g$; it is a subring of $k_q(x,y)^G$, and the next step is to show that these two rings are actually equal.  We can do this by checking that $[k_q(x,y) : R] = 2$, since $R \subseteq k_q(x,y)^G \subsetneq k_q(x,y)$ will then imply $R = k_q(x,y)^G$.  

\begin{lemma}\label{computation_lemma}
\begin{enumerate}[(i)]
 \item The following elements are all in $R$:
 \[x + \Lambda x^{-1}, \quad y - y^{-1}, \quad xy - \Lambda x^{-1}y^{-1}.\]
\item Let $b = y+y^{-1}$ as in \eqref{abc_defs}.  Then $b^2 \in R$ and $R\langle b \rangle = k_q(x,y)$.
\end{enumerate}
\end{lemma}
\begin{proof} $(i)$ 
 We begin by proving directly that $y-y^{-1} \in R$, as the others will follow easily from this.  Indeed, we will show that
\begin{equation*}y-y^{-1} = (hg-1)^{-1}(qg^2-h^2).\end{equation*}
Using the definitions in \eqref{abc_defs} this is equivalent to checking that
\begin{equation}\label{computation_2}(ab^{-1}c - b)(y-y^{-1}) = qcb^{-1}c - ab^{-1}a.\end{equation}
Expanding out the components on the right in \eqref{computation_2}, we get
\begin{equation*}\begin{aligned}qcb^{-1}c &= q(xy + \Lambda x^{-1}y^{-1})(y+y^{-1})^{-1}(xy + \Lambda x^{-1}y^{-1}) \\
&= \Big (q^2y(qy + q^{-1}y^{-1})^{-1}x + q^2y^{-1}\Lambda (qy^{-1} + q^{-1}y)^{-1}x^{-1}\Big)(xy + \Lambda x^{-1}y^{-1}) \\
&= q^2(qy + q^{-1}y^{-1})^{-1}yx^2y + q^2\alpha(\Lambda)(qy + q^{-1}y^{-1})^{-1} \\
& \quad \qquad+ q^2\Lambda(qy^{-1} + q^{-1}y)^{-1} + q^2\Lambda \alpha^{-1}(\Lambda)(qy^{-1} + q^{-1}y)^{-1}y^{-1}x^{-2}y^{-1} \\
& \\
ab^{-1}a &= (x-\Lambda x^{-1})(y+y^{-1})^{-1}(x-\Lambda x^{-1}) \\
&= \Big((qy + q^{-1}y^{-1})^{-1}x - \Lambda (q^{-1}y + qy^{-1})^{-1}x^{-1}\Big)(x-\Lambda x^{-1}) \\
&= q^2(qy + q^{-1}y^{-1})^{-1}yx^2y^{-1} - \alpha(\Lambda)(qy + q^{-1}y^{-1})^{-1} \\
&\quad \qquad - \Lambda(q^{-1}y + qy^{-1})^{-1} + q^2\Lambda \alpha^{-1}(\Lambda)(q^{-1}y + qy^{-1})^{-1}y^{-1} x^{-2}y
\end{aligned}\end{equation*}
so that the difference $qcb^{-1}c - ab^{-1}a$ is
\begin{equation}\begin{aligned}\label{computation_RHS}
qcb^{-1}c - ab^{-1}a &= q^2y(qy + q^{-1}y^{-1})^{-1}x^2(y-y^{-1}) \\
&\quad \qquad  -q^2\Lambda\alpha^{-1}(\Lambda)y^{-1}(qy^{-1} + q^{-1}y)^{-1}x^{-2}(y-y^{-1})\\
&\quad \qquad + (q^2+1)\Big(\alpha(\Lambda)(qy + q^{-1}y^{-1})^{-1} + \Lambda(qy^{-1} + q^{-1}y)^{-1}\Big)
\end{aligned}\end{equation}
Meanwhile, using the expression for $ab^{-1}c$ obtained in \eqref{abc}, we find that
\begin{equation}\begin{aligned}\label{computation_LHS}
& (ab^{-1}c - b)(y-y^{-1}) =\\
&\qquad \qquad q^2y(qy + q^{-1}y^{-1})^{-1}x^2(y-y^{-1}) \\
& \qquad \qquad \quad - q^2\Lambda \alpha^{-1}(\Lambda)y^{-1}(q^{-1}y + qy^{-1})^{-1}x^{-2}(y-y^{-1}) \\
&\, \qquad \qquad \quad+ \Big(\alpha(\Lambda)(qy + q^{-1}y^{-1})^{-1}y^{-1} - \Lambda(q^{-1}y + qy^{-1})^{-1}y - y - y^{-1}\Big)(y-y^{-1}) \\
\end{aligned}\end{equation}
Comparing the expressions in \eqref{computation_RHS} and \eqref{computation_LHS} it is immediately clear that the terms involving $x^2$ and $x^{-2}$ are equal.  This leaves just the terms involving only powers of $y$ to check; each of these are elements of $k(y)$ and therefore commutative, so it is now a simple computation to check that both expressions reduce to the form
\begin{equation*}
 \frac{(1+q)(y-y^{-1})(y+y^{-1})(q+q^{-1})}{(qy + q^{-1}y^{-1})(qy^{-1} + q^{-1}y)}.
\end{equation*}
Thus $(ab^{-1}c - b)(y-y^{-1}) = qcb^{-1}c - ab^{-1}a$ as required.  This proves that $y-y^{-1} \in R$.

Inside $k_q(x,y)$ we can notice that
\begin{equation}\begin{aligned}\label{formula_for_x}
 y^{-1}h + q^{-1}g &= (y+y^{-1})^{-1}\left(y^{-1}(x-\Lambda x^{-1}) + q^{-1}(xy + \Lambda x^{-1}y^{-1})\right) \\ 
&= (y+y^{-1})^{-1}\left(y^{-1}x - \Lambda y^{-1}x^{-1} + yx + \Lambda y^{-1}x^{-1}\right) \\ 
&= (y+y^{-1})^{-1}(y^{-1} + y)x \\
&= x.\end{aligned}
\end{equation}
and so $\Lambda x^{-1} = \varphi(y^{-1}h + q^{-1}g) = q^{-1}g - yh$.  Putting these together we get
\[x+\Lambda x^{-1} = (y^{-1}-y)h + 2q^{-1}g \in R\]
and similarly,
\[xy - \Lambda x^{-1}y^{-1} = 2qh + (y-y^{-1})g \in R.\]

$(ii)$ It's clear that $b = y+y^{-1} \not\in R$ since $R$ is a subring of $k_q(x,y)^{\varphi}$ and $b$ is not fixed by $\varphi$.  However, $(y+y^{-1})^2 = (y-y^{-1})^2 + 4$, hence $b^2 \in R$ by (i).

To prove that $R\langle b \rangle  = k_q(x,y)$ it is enough to show that $x, y \in R\langle b \rangle$.  This is now clear, however, since $y = \frac{1}{2}(y-y^{-1}) + \frac{1}{2}(y+y^{-1}) \in R \langle b \rangle$, hence by \eqref{formula_for_x}, $x = y^{-1}h + q^{-1}g \in R\langle b \rangle$ as well.
\end{proof}
Since we are working with fixed rings, the language of Galois theory is a natural choice to use here, and in \cite[\S3.6]{skewfields} we find conditions for a quotient of a general Ore extension $R[u;\gamma, \delta]/(u^2+\lambda u  + \mu)$ to be a quadratic division ring extension of $R$.  (Note that the language of \cite{skewfields} is that of \textit{right} Ore extensions, so we make the necessary adjustments below to apply the results to left extensions.)

When char $k$ $\neq 2$, such an extension will be Galois if and only if $\delta$ is inner \cite[Theorem~3.6.4(i)]{skewfields} so here it is sufficient to only consider the case when $\delta = 0$.  Further, since $b^2 \in R$ by Lemma~\ref{computation_lemma} (ii), we see that $b$ satisfies a quadratic equation over $R$ with $\lambda = 0$, which allows us to simplify matters even further.

The next result is a special case of \cite[Theorem~3.6.1]{skewfields}, which by the above discussion is sufficient for our purposes.
\begin{proposition}\label{gth_prop}
Let $K$ be a skewfield, $\gamma$ an endomorphism on $K$ and $\mu \in K^{\times}$.  The ring $T:= K[u; \gamma]/(u^2+\mu)$ is a quadratic division ring extension of $K$ if and only if T has no zero-divisors and $\mu$, $\gamma$ satisfy the following two conditions:
\begin{enumerate}
\item $\mu r =  \gamma^2(r)\mu$ for all $r \in K$;
\item $\gamma(\mu) = \mu$.
\end{enumerate}
\end{proposition} 
\begin{proof}
By \cite[Theorem~3.6.1]{skewfields} and replacing right Ore extensions with left, the ring $K[u;\gamma, \delta]/(u^2 + \lambda u + \mu)$ is a quadratic extension of $K$ if and only if it contains no zero divisors and $\gamma$, $\delta$, $\lambda$ and $\mu$ satisfy the equalities
\begin{equation*}\label{gth_conditions_1}\begin{aligned}
\gamma\delta(r) + \delta\gamma(r) &= \gamma^2(r)\lambda - \lambda\gamma(r),\\
\delta^2(r) +\lambda\delta(r) &=  \gamma^2(r)\mu - \mu r, \\
\delta(\lambda) &= \mu - \gamma(\mu) - (\lambda - \gamma(\lambda))\lambda, \\
\delta(\mu) &= (\lambda - \gamma(\lambda))\mu.
\end{aligned}\end{equation*}
Once we impose the conditions $\delta = 0$, $\lambda = 0$ the result follows immediately.
\end{proof}
Viewing $R$ as a subring of $k_q(x,y)$, we can set $u = b$, $\mu = -b^2$.  The following choice of $\gamma$ is suggested by \cite{division}.
\begin{lemma}\label{gamma_def}
Let $b$, $h$ and $g$ be as defined in \eqref{abc_defs}, and $R$ the division ring generated by $h$ and $g$ inside $k_q(x,y)$.  Then the conjugation map defined by
\begin{equation*}
\gamma(r) = brb^{-1} \quad \forall r \in R
\end{equation*}
is a well-defined automorphism on $R$.
\end{lemma}
\begin{proof}
It is sufficient to check that the images of the generators of $R$ under $\gamma$ and $\gamma^{-1}$ are themselves in $R$, i.e. that
\begin{align*}
\gamma(h) &= (ab)b^{-2} &\gamma(g) = (cb)b^{-2} \\
 \gamma^{-1}(h) &= b^{-2}(ab) &\gamma^{-1}(g) = b^{-2}(cb)
\end{align*}
are all in $R$.

By Lemma \ref{computation_lemma} (ii) we already know that $b^2  \in R$.  As for $ab$ and $cb$, they decompose into elements of $R$ as follows:
\begin{equation}\begin{aligned}\label{formulas_cb}
 ab &= (x-\Lambda x^{-1})(y+y^{-1}) \\
&= xy + xy^{-1} - \Lambda x^{-1}y - \Lambda x^{-1}y^{-1} \\
&= 2(xy - \Lambda x^{-1}y^{-1}) - (x+\Lambda x^{-1})(y-y^{-1})\in R\\
&\\
 cb &= (xy + \Lambda x^{-1}y^{-1})(y+y^{-1}) \\
&= xy^2 + x + \Lambda x^{-1} + \Lambda x^{-1}y^{-2} \\
&= (xy - \Lambda x^{-1}y^{-1})(y-y^{-1}) + 2(x+\Lambda x^{-1}) \in R
\end{aligned}\end{equation}
by Lemma~\ref{computation_lemma} (i).  Therefore $\gamma$ is a well-defined bijection on $R$, and since conjugation respects the relation $hg - qgh = 1-q$, it is an automorphism on $R$.\end{proof}
We are now in a position to prove Theorem~\ref{epic_theorem}.  

Recall that $R \subseteq k_q(x,y)^G$ is a division ring with generators $h$ and $g$, which satisfy a quantum Weyl relation $hg - qgh = 1-q$.  We can make a change of variables $h \mapsto \frac{1}{1-q}(hg-gh)$ so that $R$ has the structure of a $q$-division ring \cite[Proposition~3.2]{AD2}.  (The only exception is when $q=1$, where this change of variables does not make sense; however, since $h$ and $g$ already ``$q$-commute'' in this case we can simply set $f:=h$.)

Define the automorphism $\gamma$ as in Lemma~\ref{gamma_def} and set $\mu := -b^2 \in R$.  The extension $L :=R[b;\gamma]/(b^2+\mu)$ is a subring of the division ring $k_q(x,y)$, and therefore has no zero divisors.  Further,
\[ \gamma^2(r) \mu = -(b^2rb^{-2})b^2 = -b^2r = \mu r \quad \forall r \in R\]
and similarly $\gamma(\mu) = \mu$.  Therefore by Proposition~\ref{gth_prop}, $L$ is a quadratic extension of $R$.  Since it is a subring of $k_q(x,y)$ containing both $R$ and $b$, by Lemma~\ref{computation_lemma} (ii) we can conclude that $L = k_q(x,y)$.

Now since $R \subseteq k_q(x,y)^G \subsetneq k_q(x,y) = L$, and the extension $R \subset L$ has degree 2, we must have $R = k_q(x,y)^G$ and Theorem~\ref{epic_theorem} is proved.

\section{Further fixed rings}\label{more fixed rings}
Theorem~\ref{epic_theorem} came about as a result of a related question, namely: if we take an automorphism of finite order defined on $k_q[x^{\pm1},y^{\pm1}]$ and extend it to $k_q(x,y)$, what does its fixed ring look like?  

As discussed in \cite[\S4.1.1]{dumas_invariants}, the automorphism group of $k_q[x^{\pm1},y^{\pm1}]$ is generated by automorphisms of scalar multiplication and the monomial automorphisms (see Definition~\ref{def_monomial_autos} below).  Since the case of scalar multiplication has been covered in Theorem~\ref{AD_prop}, in this section we will focus on monomial automorphisms with the aim of proving Theorem~\ref{thm_monomial_results}.

For the remainder of this section $q\in k^{\times}$ can still be any non-zero scalar, but we will assume that $k$ contains a square root of $q$, denoted by $\hat{q}$.  

\begin{definition}\label{def_monomial_autos} We call an automorphism of $k_q[x^{\pm1},y^{\pm1}]$ or $k_q(x,y)$ a \textit{monomial automorphism} if it can be represented as an element of $SL_2(\mathbb{Z})$ as follows: for $g = \left(\begin{smallmatrix}a&b\\c&d\end{smallmatrix}\right) \in SL_2(\mathbb{Z})$,
\[g.y =  \hat{q}^{ac}y^ax^c, \quad g.x =  \hat{q}^{bd}y^bx^d \quad a,b,c,d \in \mathbb{Z}\]
(see \cite[\S1.3]{Baudry}, with the roles of $x$ and $y$ exchanged).  
\end{definition}
It is well known that up to conjugation, $SL_2(\mathbb{Z})$ has only four non-trivial finite subgroups: the cyclic groups of orders 2, 3, 4 and 6 (see, for example, \cite[\S1.10.1]{Lorenz}).   Table~\ref{table_of_maps} lists conjugacy class representatives for each of these groups, and we will use the same symbols to refer to both these automorphisms and their extensions to $D$.
\begin{table}
\centering
\begin{tabular}[h]{c|rl}
Order & \multicolumn{2}{c}{Automorphism} \\ \hline
2 & $\tau:$&$ x \mapsto x^{-1},\  y \mapsto y^{-1}$ \\
3 & $\sigma:$&$ x \mapsto y, \ y \mapsto \hat{q}y^{-1}x^{-1}$ \\
4 & $\rho:$&$ x \mapsto y^{-1}, \ y \mapsto x$ \\
6 & $\eta:$&$ x \mapsto y^{-1}, \ y \mapsto \hat{q}yx$
\end{tabular}
\caption{Conjugacy class representatives of finite order monomial automorphisms on $k_q[x^{\pm1},y^{\pm1}]$.}\label{table_of_maps}
\end{table}

As noted in \cite[\S1.3]{Baudry}, it is sufficient to consider the fixed rings for one representative of each conjugacy class. We will therefore approach Theorem~\ref{thm_monomial_results} by examining the fixed rings of $D$ under each of the automorphisms in Table~\ref{table_of_maps} in turn.

By Theorem~\ref{order_2_monomial_result}, we already know that $D^{\tau} \cong D$.  This is proved by methods from noncommutative algebraic geometry in \cite[\S13.6]{SV1}, but the authors also provide a pair of $q$-commuting generators for $D^{\tau}$, namely
\begin{equation}\label{order_2_gens}
 u = (x-x^{-1})(y^{-1}-y)^{-1}, \quad v = (xy - x^{-1}y^{-1})(y^{-1}-y)^{-1}.
\end{equation}
We can use this and Theorem~\ref{epic_theorem} to check that the fixed ring of $D$ under an order 4 monomial automorphism is again isomorphic to $D$.
\begin{theorem}\label{order_4_thm}
 Let $\rho$ be the order 4 automorphism on $D$ defined by
\[\rho: x \mapsto y^{-1}, \quad y \mapsto x.\]
Then $D^{\rho} \cong D$ as $k$-algebras.
\end{theorem}
\begin{proof}
We can first notice that $\rho^2 = \tau$, so the fixed ring $D^{\rho}$ is a subring of $D^{\tau}$.  By \cite[\S13.6]{SV1}, $D^{\tau} = k_q(u,v)$ with $u,v$ as in \eqref{order_2_gens}, so it is sufficient to consider the action of $\rho$ on $u$ and $v$.  By direct computation, we find that
\begin{equation*}
 \begin{aligned}
  \rho(u) &= (y^{-1} - y)(x^{-1} - x)^{-1} = -u^{-1}\\
\rho(v) &= (y^{-1}x - yx^{-1})(x^{-1} - x)^{-1} = (u^{-1} - qu)v^{-1}
 \end{aligned}
\end{equation*}
i.e. $\rho$ acts as $\varphi$ from \eqref{def_varphi} on $k_{q^{-1}}(v,u)$, which by Theorem~\ref{AD_prop} is isomorphic to $k_q(u,v)$.  Now by Theorem~\ref{epic_theorem}, $D^{\rho} \cong D^{\varphi} \cong D$.\end{proof}

We now turn our attention to the fixed ring of $D$ under the order 3 automorphism $\sigma$ defined in Table~\ref{table_of_maps}, where matters become significantly more complicated.  Attempting to construct generators by direct analogy to the previous cases fails, and computations become far more difficult as both $x$ and $y$ appear in the denominator of any potential generator.  While the same theorem can be proved for this case, our chosen generators are unfortunately quite unintuitive.

For the following results, we will assume that $k$ contains a primitive third root of unity, denoted $\omega$.  As with Theorem~\ref{epic_theorem}, we define certain elements which are fixed by $\sigma$ or are acted upon as multiplication by a power of $\omega$.  We set
\begin{equation}\label{order_3_building_blocks}\begin{aligned}
 a &= x + \omega y + \omega^2 \hat{q}y^{-1}x^{-1} \\
b &= x^{-1} + \omega y^{-1} + \omega^2 \hat{q}yx \\
c &= y^{-1}x + \omega \hat{q}^3y^2x + \omega^2\hat{q}^3y^{-1}x^{-2} \\
\pa &= x + y + \hat{q}y^{-1}x^{-1} \\
\pb &= x^{-1} + y^{-1} + \hat{q} yx \\
\pc &= y^{-1}x + \hat{q}^3y^2x + \hat{q}^3y^{-1}x^{-2}
\end{aligned}\end{equation}
The elements $\pa$, $\pb$ and $\pc$ are fixed by $\sigma$, while $\sigma$ acts on $a$, $b$ and $c$ as multiplication by $\omega^2$.  We can further define
\begin{equation}\label{order_3_gens}\begin{aligned}
 g &= a^{-1}b \\
 f &= \pb - \omega^2\pa g + (\omega^2-\omega)\hat{q}^{-1}(\omega^2g^2 + \hat{q}^2g^{-1})
\end{aligned}\end{equation}
\begin{proposition}\label{order_3_q_comm_proof}
Let $k$ be a field of characteristic 0 that contains a primitive third root of unity $\omega$ and a square root of $q$, denoted by $\hat{q}$.  The elements $f$ and $g$ in \eqref{order_3_gens} are fixed by $\sigma$ and satisfy $fg = qgf$.
\end{proposition}
\begin{proof}
 As always the first statement is clear: $\sigma$ acts on $a$ and $b$ by $\omega^2$ and therefore fixes $g$, and since $\pa$ and $\pb$ are already fixed by $\sigma$ we can now see that $\sigma(f) = f$.

To verify the second statement, we need to understand how $g$ interacts with $\pa$ and $\pb$.  Simple multiplication of polynomials yields the identities
\begin{align*}
 a\pa &=\pa a + (\omega-\omega^2)(\hat{q}-\hat{q}^{-1})b \\
a \pb  &= \hat{q}^2\pb a + (\hat{q}^{-2} - \hat{q}^2)c \\
\pa b &= \hat{q}^2b\pa  + \omega(\hat{q}^{-2} - \hat{q}^2)c \\
\pb b &= b\pb  + (\omega^2-\omega)(\hat{q}-\hat{q}^{-1})a
\end{align*}
and hence
\begin{align*}
 g\pa  &= \hat{q}^{-2}\pa g - \hat{q}^{-2}\omega(\hat{q}^{-2}-\hat{q}^2)a^{-1}c - (\omega-\omega^2)\hat{q}^{-2}(\hat{q}-\hat{q}^{-1})g^2 \\
g \pb  &= \hat{q}^{-2}\pb g - (\omega^2-\omega)(\hat{q}-\hat{q}^{-1}) - \hat{q}^{-2}(\hat{q}^{-2} - \hat{q}^2)a^{-1}cg
\end{align*}
since $g = a^{-1}b$.

Now by direct computation, we find that
\begin{align*}
 \hat{q}^2gf &= \hat{q}^2g\pb - \hat{q}^2\omega^2g\pa g + (w^2-w)\hat{q}(\omega^2g^3 + \hat{q}^2)\\
&= \pb g - (\omega^2-\omega)(\hat{q}-\hat{q}^{-1})\hat{q}^2 - (\hat{q}^{-2} - \hat{q}^2)a^{-1}cg \\
& \qquad - \omega^2\pa g^2 + (\hat{q}^{-2}-\hat{q}^2)a^{-1}cg + \omega^2(\omega-\omega^2)(\hat{q}-\hat{q}^{-1})g^3 \\
& \qquad + (\omega^2-\omega)\hat{q}(\omega^2g^3 + \hat{q}^2) \\
&= \pb g - \omega^2 \pa g^2 + (\omega^2-\omega)\hat{q}^{-1}(\omega^2g^3 + \hat{q}^2)\\
&=fg\qedhere
\end{align*}
\end{proof}
\begin{theorem}\label{epic_theorem_2}
Let $k$, $f$ and $g$ be as in Proposition~\ref{order_3_q_comm_proof}.  Then the division ring $k_q(f,g)$ generated by $f$ and $g$ over $k$ is equal to the fixed ring $D^{\sigma}$, and hence $D^{\sigma} \cong D$ as $k$-algebras.
\end{theorem}
\begin{proof}
We claim that it suffices to prove $k_q[x^{\pm1},y^{\pm1}]^{\sigma} \subset k_q(f,g)$.  Indeed, $k_q[x^{\pm1},y^{\pm1}]$ is a Noetherian domain, and therefore both left and right Ore, while $\langle \sigma \rangle$ is a finite group.  We can therefore apply \cite[Theorem~1]{faith} to see that
\[Q\big(k_q[x^{\pm1},y^{\pm1}]^{\sigma}\big) = k_q(x,y)^{\sigma}\]
where $Q(R)$ denotes the full ring of fractions of a ring $R$.  Hence if $k_q[x^{\pm1},y^{\pm1}]^{\sigma} \subset k_q(f,g)$, we see that
\[Q(k_q[x^{\pm1},y^{\pm1}]^{\sigma}) \subseteq k_q(f,g) \subseteq k_q(x,y)^{\sigma} \Rightarrow k_q(f,g) = k_q(x,y)^{\sigma}.\]

We will show that $k_q[x^{\pm1}, y^{\pm1}]^{\sigma}$ is generated as an algebra by the elements $\pa$, $\pb$ and $\pc$ from \eqref{order_3_building_blocks}, and then check that these three elements are in $k_q(f,g)$. 

By \cite[Th\'eor\`eme~2.1]{Baudry}, $k_q[x^{\pm1}, y^{\pm1}]^{\sigma}$ is generated as a Lie algebra with respect to the commutation bracket by seven elements:
\begin{gather*}
R_{0,0} = 1, \quad R_{1,0} = x+ y + \hat{q}y^{-1}x^{-1}, \quad R_{1,1} = x^{-1} + y^{-1} + \hat{q}yx, \\
R_{1,2} = y^{-1}x + \hat{q}^3y^2x + \hat{q}^3y^{-1}x^{-2}, \quad R_{1,3} = y^{-1}x^2 + \hat{q}^5y^3x + \hat{q}^8y^{-2}x^{-3}, \\
R_{2,0} = x^2+y^2+\hat{q}^4y^{-2}x^{-2}, \quad R_{3,0} = x^3 + y^3 + \hat{q}^9y^{-3}x^{-3}.
\end{gather*}
and so it is also generated as a $k$-algebra by these elements.  $R_{1,0}$, $R_{1,1}$ and $R_{1,2}$ are precisely the aforementioned elements $\pa$, $\pb$ and $\pc$, and it is a simple computation to verify that $R_{1,3}$, $R_{2,0}$ and $R_{3,0}$ are in the algebra generated by these three. 

It is clear from the definition of $f$ that once we have found either $\pa$ or $\pb$ in $k_q(f,g)$ we get the other one for free, and we can also observe that
\[\pa \pb - \hat{q}^2 \pb \pa = (\hat{q}^{-2} - \hat{q}^2) \pc - 3\hat{q}^2 + 3 \in k_q(f,g)\]
so $\pc \in k_q(f,g)$ follows from $\pa$, $\pb \in k_q(f,g)$.  Unfortunately there seems to be no easy way to make the first step, i.e. verify that either $\pa$ or $\pb$ is in $k_q(f,g)$. 

In fact, the element $\pa$ can be written in terms of $f$ and $g$ as in the following equality; this is the result of a long and tedious calculation, and was verified using the computer algebra system Magma (v2.18).  We find that
\begin{gather*}
\pa = (\omega - \omega^2)^{-1}\hat{q}^{-2}g^{-1}f + (\omega^2\hat{q} + \omega\hat{q}^{-1})g + (\hat{q} + \hat{q}^{-1})g^{-2}  \\
\qquad \qquad    + (\omega - \omega^2)\Big(\hat{q}^{-2}g^3 + (\hat{q}^2+1) + \hat{q}^4g^{-3}\Big)f^{-1}.
\end{gather*}
Therefore $\pa \in k_q(f,g)$, and the result now follows.
\end{proof}
\begin{remark}
By analogy to the pairs of generators in \eqref{order_2_gens} and Theorem~\ref{epic_theorem}, we might hope to find similarly intuitive generators for $k_q(x,y)^{\sigma}$.  Having set $g := a^{-1}b$ as in the proof above, computation in Magma shows that there does exist a left fraction $f' \in k_q(x,y)^{\sigma}$ such that $f'g = qgf'$; unfortunately, $f'$ takes 9 pages to write down.  In the interest of brevity, we chose to use the less intuitive $f$ defined in \eqref{order_3_gens} instead.
\end{remark}

In a similar manner to Theorem~\ref{order_4_thm}, we can now describe the one remaining fixed ring $D^{\eta}$ using our knowledge of the fixed rings with respect to monomial maps of order 2 and 3.
\begin{theorem}\label{order_6_thm}
Let $\eta$ be the order 6 map defined in Table~\ref{table_of_maps}, and suppose $k$ contains a primitive third root of unity and both a second and third root of $q$.  Then $ D^{\eta} \cong D$ as $k$-algebras.
\end{theorem}
\begin{proof}
We first note that $\eta^3 = \tau$, so $D^{\eta} = (D^{\tau})^{\eta}$.  Take $u,v$ in \eqref{order_2_gens} as our generators of $D^{\tau}$, and now we can observe that the action of $\eta$ on $u$ and $v$ is as follows:
\begin{align*}
\eta(u) &= (y^{-1} - y)(\hat{q}^{-1}x^{-1}y^{-1} - \hat{q}yx)^{-1} \\
&=-\hat{q}(y^{-1} - y)(xy -x^{-1}y^{-1})^{-1} \\
&=-\hat{q}v^{-1} \\
\eta(v) &= (\hat{q}y^{-1}yx - \hat{q}^{-1}yx^{-1}y^{-1})(\hat{q}^{-1}x^{-1}y^{-1} - \hat{q}yx)^{-1} \\
&= -q(x-x^{-1})(xy-x^{-1}y^{-1})^{-1} \\
&= -v^{-1}u 
\end{align*}
Let $p = \sqrt[3]{q^{-1}}$. By making a change of variables $u_1= -p^{-1}\hat{q}^{-1}u$, $v_1 =  pv$ in $D^{\tau} = k_q(u,v)$, we see that $\eta$ acts on $D^{\tau}$ as
\[\eta(u_1) = v_1^{-1}, \quad \eta(v_1) = \hat{q}^{-1}v_1^{-1}u_1\]
This is a monomial map of order 3 and so its fixed ring is isomorphic to $D^{\sigma}$, as noted in \cite[\S1.3]{Baudry}.  Now by Theorem~\ref{epic_theorem_2} and Theorem~\ref{order_2_monomial_result}, $D^{\eta} = (D^{\tau})^{\eta} \cong (D^{\tau})^{\sigma} \cong D^{\tau} \cong D$.
\end{proof}
\begin{theorem}\label{thm_monomial_results_2}
Let $k$ be a field of characteristic zero, containing a primitive third root of unity $\omega$ and both a second and a third root of $q$.  If $G$ is a finite group of monomial automorphisms of $D$ then $D^G \cong D$ for the same choice of $q$.
\end{theorem}
\begin{proof}
Theorems~\ref{order_4_thm}, \ref{epic_theorem_2}, \ref{order_6_thm} and \cite[\S13.6]{SV1}.
\end{proof}

\section{Consequences for the automorphism group of $D$}\label{consequences}

The construction of $q$-commuting pairs of elements is closely linked to questions about the automorphisms and endomorphisms of the $q$-division ring: such maps are defined precisely by where they send the two generators of $D$, and naturally these images must $q$-commute.  Despite similarities to the commutative field $k(x,y)$ a full description of the automorphism group $Aut(D)$ remains unknown, with a major stumbling block being understanding the role played by conjugation maps.

Intuition suggests that ``inner automorphism'' and ``conjugation'' should be synonymous; certainly all conjugation maps should be bijective, at the very least.  Here we challenge this intuition by showing that the conjugation maps defined in \cite{AC1} not only gives rise to conjugations which are not inner, but also conjugation maps which are well-defined \textit{endomorphisms} (not automorphisms) on $D$.  This provides answers to several of the questions posed at the end of \cite{AC1} (outlined in Questions~\ref{AC_questions} below), while also raising several new ones.

For the remainder of this section, we require that $q \in k^{\times}$ is not a root of unity, that is $q^n \neq 1$ for all $n \geq 1$.

Let $X$, $Y$ be a pair of $q$-commuting generators for $D$.  As in \cite{AC1}, we call an automorphism of $D$ \textit{elementary} if it has one of the following forms:
\begin{equation}\label{elementary_autos}\begin{aligned}
\tau:&{}\quad X\mapsto X^{-1}, \ Y \mapsto Y^{-1} \\
h_X:&{}\quad X \mapsto b(Y)X, \ Y \mapsto Y,\quad b(Y) \in k(Y)^{\times}\\
h_Y:&{}\quad X \mapsto X,\ Y \mapsto a(X)Y,\quad a(X) \in k(X)^{\times} 
\end{aligned}\end{equation}
In \cite{AC1}, the authors prove that any endomorphism $\psi$ of $k_q(X,Y)$ can be decomposed into a product of maps
\begin{equation}\label{decomp}\psi = \varphi_1\dots \varphi_n c_{z^{-1}} \tau^{\epsilon}.\end{equation}
Here the $\varphi_i$ are elementary automorphisms and $\tau$ is as defined in \eqref{elementary_autos}, while $\epsilon \in \{0,1\}$.  The map $c_{z^{-1}}(r) = z^{-1}rz$ $\forall r \in k_q(X,Y)$ represents conjugation by some element $z^{-1} \in k_q(Y)(\!(X)\!)$. Here $z$ is defined recursively as follows.
\begin{proposition} \cite[Proposition~3.3]{AC1}\label{AC_theorem_recap} Let $F$, $G \in k_q(Y)(\!(X)\!)$ be $q$-commuting elements of the form
\begin{equation}\label{standard_form}F = f_sX^{s} + \sum_{i > s} f_i X^i, \quad G = \lambda Y^{s} + \sum_{i> 0} g_i X^i\end{equation}
with $s \in \{1,-1\}$, $\lambda \in k^{\times}$ and $f_i$, $g_i \in k(Y)$.  Then there exists an element $z \in k_q(Y)(\!(X)\!)$ defined by
\begin{equation}\label{def_z}\begin{gathered}z_0 = 1; \quad z_n = Y^{-s}(1-q^s)^{-1}\left(g_n + \sum_{\stackrel{i + j = n}{i,j>0}}z_j\alpha^j(g_i)\right) \textrm{ for } n\geq 1; \\
z: = \sum_{i \geq 0}z_i X^i. \end{gathered}
\end{equation}
such that
\[zFz^{-1} = f_sX^s, \quad zGz^{-1} = \lambda Y^s \]
\end{proposition}
The following questions are posed by Artamonov and Cohn in \cite{AC1}.
\begin{questions}\label{AC_questions}\ 
\begin{enumerate}
 \item Does there exist an element $z$ satisfying the recursive definition \eqref{def_z}, such that $z \not\in k_q(X,Y)$?  What if $z^n \in k_q(X,Y)$ for some positive integer $n$?
 \item Does there exist an element $z$ from \eqref{def_z} such that $z^{-1}k_q(X,Y)z \subsetneq k_q(X,Y)$?
 \item The group of automorphisms of $k_q(X,Y)$ is generated by elementary automorphisms, conjugation by some elements of the form $z$, and $\tau$.  Find a set of defining relations for this generating set.
\end{enumerate}
\end{questions}
We first note that (3) needs rephrasing, since we can provide affirmative answers for both (1) and (2).  Indeed, we will construct examples of conjugation automorphisms $c_z$ satisfying $z^2 \in D$ (Proposition~\ref{z_automorphism_order_2}) and $z^n \not\in D$ for all $n \geq 1$ (Proposition~\ref{z_automorphism_inf_order}), and also a conjugation endomorphism such that $z^{-1}Dz \subsetneq D$ (Proposition~\ref{z_endomorphism}).

In light of this, (3) should be modified to read:
\begin{enumerate}\setcounter{enumi}{3}
\item Under what conditions is $c_z$ an automorphism of $D$ rather than an endomorphism?  Using this, give a set of generators and relations for $Aut(D)$.
\end{enumerate}
For each of our examples below, we start by defining a homomorphism $\psi$ on $D$ and then verify that the image of the generators of $D$ under $\psi$ have the form \eqref{standard_form}.  This allows us to use Proposition~\ref{AC_theorem_recap} to construct $z$ as in \eqref{def_z} such that $\psi = c_{z^{-1}}$ (possibly after a change of variables in $D$ to ensure the leading coefficients are both 1).  The final step in each proof is checking whether $z \in D$, for which we will use the following lemmas.
\begin{lemma}\label{two_conjugations}
 Let $z_1$, $z_2 \in k_q(Y)(\!(X)\!)$.  If the conjugation maps $c_{z_1}$ and $c_{z_2}$  have the same action on $D = k_q(X,Y)$, then $z_1$ and $z_2$ differ only by a scalar.
\end{lemma}
\begin{proof}
 If $c_{z_1} = c_{z_2}$, then $z_1Yz_1^{-1} = z_2Yz_2^{-1}$, i.e. $z_2^{-1}z_1Y = Yz_2^{-1}z_1$.  Similarly $z_2^{-1}z_1$ commutes with $X$, so $z_2^{-1}z_1$ is in the centralizer of $D$ in $k_q(Y)(\!(X)\!)$, which we write as $C(D)$.

We now verify that $C(D) = k$.  Indeed, if $u = \sum_{i \geq n}a_i X^i \in C(D)$, then $u$ must commute with $Y$, i.e.
\[Y \sum_{i \geq n} a_i X^i = \sum_{i \geq n} a_i X^i Y = \sum_{ i\geq n}q^ia_i Y X^i\]
Since $q$ is not a root of unity, we must have $a_i =0$ for all $i \neq 0$, i.e. $u = a_0 \in k(Y)$.  Since $u$ is now in $D$ and must commute with both $X$ and $Y$, $u \in Z(D) = k$.  The result now follows.
\end{proof}
Recall that for $r \in k_q[X,Y]$, $deg_X(r)$ denotes the degree of $r$ as a polynomial in $X$.  This extends naturally to a notion of degree on $k_q(X,Y)$ by defining 
\[deg_X(t^{-1}s) := deg_X(s) - deg_X(t),\]
where $s, t \in k_q[X,Y]$.  We note that this definition is multiplicative.
\begin{lemma}\label{matching_degrees} If $c_z$ is an inner automorphism on $k_q(X,Y)$, then $c_z(Y) = v^{-1}u$ satisfies $deg_X(u) = deg_X(v)$.\end{lemma}
\begin{proof} We can write the commutation relation in $D$ as $YX = \beta(X)Y$, where $\beta$ is the automorphism $X \mapsto q^{-1}X, \ Y \mapsto Y$.  Let $c_z$ be an inner automorphism on $D$, so that $z = t^{-1}s \in D$ for $s, t \in k_q[X,Y]$.  Thus the image of $Y$ under $c_z$ is
 \begin{equation}\label{ore_1}c_z(Y) = t^{-1}s\beta(s)^{-1}\beta(t) Y\end{equation}
Since $\beta$ does not affect the $X$-degree of a polynomial and $deg_X$ is multiplicative, it is clear from \eqref{ore_1} that $deg_X(c_z(y)) = 0$ and hence $deg_X(u) = deg_X(v)$ as required.
\end{proof}

We are now in a position to answer Questions~\ref{AC_questions} (1) and (2).

\begin{proposition}\label{z_automorphism_order_2}
Let $D$ and $\varphi$ be as in Theorem~\ref{epic_theorem}, and set $E = D^{\varphi}$.  With an appropriate choice of generators for $E$, the map
\[\gamma: \ E \rightarrow E:\  r \mapsto (y^{-1} + y)r(y^{-1} + y)^{-1}\]
defined in Lemma~\ref{gamma_def} is an automorphism of the form $c_{z^{-1}}$, with $z$ defined as in \eqref{def_z}. Further, we have $z\not\in E$ while $z^2 \in E$.  This provides an affirmative answer to Question~\ref{AC_questions} (1).
\end{proposition}
\begin{proof}
$E$ is a $q$-division ring by Theorem~\ref{epic_theorem}, and $\gamma$ is an automorphism by Lemma~\ref{gamma_def}.  Write $E = k_q(f,g)$, where $f:=\frac{1}{1-q}(hg-gh)$ as in the proof of Theorem~\ref{epic_theorem}, which allows us to use the methods of \cite{AC1}.

In order to check that $\gamma$ has the form described by \eqref{def_z}, it is sufficient to check that $\gamma(f)$, $\gamma(g)$ are of the form \eqref{standard_form}.  Recall that $\gamma(g) = (cb)b^{-2}$, which can be written in terms of $h$ and $g$ using Lemma~\ref{computation_lemma} and \eqref{formulas_cb}.  From the definition $f = \frac{1}{1-q}(hg - gh)$ we find that $h = g^{-1}(1-f)$, and after some rearranging we obtain
\begin{align*}
 b^2 &= (q^3 - g^4)(q^7-g^4)q^{-6}g^{-3}f^{-2} + [\textrm{higher terms in }f] \\
 cb &= (q^3 - g^4)(q^7-g^4)q^{-7}g^{-3}f^{-2} + [\textrm{higher terms in }f]
\end{align*}
Therefore the lowest term of $\gamma(g) = (cb)b^{-2}$ is $qg$, as required.

By \cite[Proposition~3.2]{AC1}, it now follows that $\gamma(f)$ must have the form
\[\gamma(f) = b_1f + \sum_{i > 1} b_i f^i, \quad b_1, b_i \in k(g).\]
We can make a change of variables in $k_q(f,g)$ using elementary automorphisms to scale $\gamma(g)$ by $q^{-1}$ and ensure that $b_1= 1$. Now by \cite[Theorem~3.5]{AC1}, $\gamma = c_{z^{-1}}$ with $z$ constructed as in \eqref{def_z}.  

By Lemma \ref{two_conjugations}, $y+y^{-1}$ and $z$ differ by at most a scalar.  Since $y + y^{-1} \not\in E$, $\gamma = c_{z^{-1}}$ defines an automorphism of $E$ with $z \not\in E$.

Finally, we have already noted in Lemma~\ref{computation_lemma} (ii) that $b^2 = (y-y^{-1})^2 + 4 \in E$, and so $z^2 \in E$ as well.
\end{proof}

\begin{remark}It is worth noting that this phenomenon of non-inner conjugations cannot happen when $q$ is a root of unity.  Indeed, if $q^n = 1$ for some $n$, then $D$ is a finite dimensional central simple algebra over its centre and by the Skolem-Noether theorem every automorphism of $D$ should be inner.  The automorphism $\gamma$ is still a well-defined automorphism on $D$ in this case, but the difference is that now $D$ has a non-trivial centre: since $y^{n} - y^{-n}$ is a central element we can replace $y + y^{-1}$ with $(y+y^{-1})(y^n - y^{-n})$ in the definition of $\gamma$ without affecting the map at all.  We now have
\[(y+y^{-1})(y^{n} - y^{-n}) = y^{n+1} - y^{-(n+1)} + y^{n-1} - y^{-(n-1)} \in k_q(x,y)^G\]
and so $\gamma$ is indeed an inner automorphism in this case. \end{remark}

\begin{proposition}\label{z_endomorphism}
Let $D = k_q(x,y)$.  Then there exists $z \in k_q(y)(\!(x)\!)$ such that $z^{-1}Dz \subsetneq D$.  This provides an affirmative answer to Question~\ref{AC_questions}(2).
\end{proposition}
\begin{proof}
We can view the isomorphism $\theta: D \stackrel{\sim}{\longrightarrow} D^G$ from Theorem \ref{epic_theorem} as an endomorphism on $D$, and so $\theta$ decomposes into the form \eqref{decomp} with $z$ constructed as in \eqref{def_z}.  Since $\theta$ is not surjective, and $c_{z^{-1}}$ is the only map in the decomposition not already known to be an automorphism, we must have $z^{-1}Dz \subsetneq D$.
\end{proof}

Propositions \ref{z_automorphism_order_2} and \ref{z_endomorphism} illustrate some of the difficulties involved in giving a set of relations for $Aut(D)$: not only is it possible for both endomorphisms and automorphisms to arise as conjugations $c_z$ for $z \not\in D$, but as we show next, it turns out to be quite easy to define an automorphism $\sigma$ on $D$ which is a product of elementary automorphisms, but also satisfies $\psi = c_{z^{-1}}$ with $z^n \not\in D$ for any $n \geq 1$.

\begin{example}\label{auto_example}We define maps by
\begin{align*}
h_1&: x \mapsto (1+y)x, \ y \mapsto y,\\
h_2&: x \mapsto x, \ y\mapsto (1+x)y,\\
h_3&: x \mapsto \frac{1}{1+y}x, \ y \mapsto y,
\end{align*}
which are all elementary automorphisms on $k_q(x,y)$.  Let $\psi = h_3\circ h_2 \circ h_1$, so that
\begin{align*}
\psi(x) &=  x + \frac{qy}{(1+y)(1+qy)}x^2,\\
\psi(y) &=  y + \frac{qy}{1+y}x.
\end{align*}
is an automorphism on $k_q(x,y)$.  These have been chosen so that $\psi(x)$, $\psi(y)$ are already in the form \eqref{standard_form}, so there exists $z \in k_q(y)(\!(x)\!)$ defined by \eqref{def_z} such that $\psi = c_{z^{-1}}$.  Since $\psi(y)$ is a polynomial in $x$ of non-zero degree, $\psi$ is not an inner automorphism by Lemma \ref{matching_degrees}.

In fact, $\psi^n(y)$ is a polynomial in $x$ of degree $n$.  The key observation in proving this is to notice that $(1+y)^{-1}x$ is fixed by $\psi$.  Indeed,
\begin{align*}
 \psi((1+y)^{-1}x) &= \left[1 + y + \frac{qy}{1+y}x\right]^{-1}\left[x + \frac{qy}{(1+y)(1+qy)}x^2\right] \\
&= \left[\left(1+\frac{qy}{(1+y)(1+qy)}x\right)(1+y)\right]^{-1}\left[1 + \frac{qy}{(1+y)(1+qy)}\right]x \\
&= (1+y)^{-1}x.
\end{align*}
If we write $\psi(y) = y(1+q(1+y)^{-1}x)$, it is now clear by induction that 
\begin{equation}\label{psi_induction}\psi^n(y) = \psi^{n-1}(y)(1+q(1+y)^{-1}x)\end{equation}
is polynomial in $x$ of degree $n$.  
\end{example}
\begin{proposition}\label{z_automorphism_inf_order}With $\psi$ as in Example~\ref{auto_example} and $D = k_q(x,y)$, $\psi = c_{z^{-1}}$ is an example of a conjugation automorphism satisfying $z^n \not\in D$ for all $n \geq 1$.
\end{proposition}
\begin{proof}
By \eqref{psi_induction}, $\psi^n(y) = z^{-n}yz^n$ is a polynomial in $x$ of degree $n$ so by Lemma~\ref{matching_degrees} we have $z^n \not\in D$ for all $n \geq 1$.
\end{proof}
Combining these results, we obtain the theorem promised in the introduction.
\begin{theorem}\label{thm_AC_results_2}
Let $k$ be a field of characteristic zero and $q \in k^{\times}$ not a root of unity.  Then:
\begin{enumerate}[(i)]
\item The $q$-division ring $D$ admits examples of bijective conjugation maps by elements $z \not\in D$; these include examples satisfying $z^n \in D$ for some positive $n$, and also those such that $z^n \not\in D$ for all $n \geq 1$.
\item $D$ also admits an endomorphism which is not an automorphism, which can be represented in the form of a conjugation map.  
\end{enumerate}
\end{theorem}
\begin{proof}
Propositions~\ref{z_automorphism_order_2}, \ref{z_endomorphism}, \ref{z_automorphism_inf_order}.
\end{proof}
We have seen that the set of generators for $Aut(D)$ proposed by Artamonov and Cohn in \cite{AC1} in fact generate the whole endomorphism group $End(D)$.  In \cite{AD3}, Alev and Dumas construct another potential set of generators for $Aut(D)$ by analogy to the commutative case, which corresponds to the group generated by the elementary automorphisms and the \textit{inner} automorphisms.  

A good test case for this proposed set of generators would be the automorphism $\gamma$ in Proposition~\ref{z_automorphism_order_2}: can it be decomposed into a product of elementary automorphisms and inner automorphisms?  The next proposition indicates one way of approaching this question.
\begin{proposition}
 Let $c_z$ be a bijective conjugation map on $k_q(x,y)$. Suppose that $c_z$ fixes some element $r \in k_q(x,y)\backslash k$, and that there is a product $\varphi$ of elementary automorphisms such that $r$ is the image of $x$ or $y$ under $\varphi$.  Then $c_z$ decomposes as a product of elementary automorphisms.
\end{proposition}
\begin{proof}
 Suppose first that $\varphi(x) = r$.  Define $u := \varphi(x)$, $v := \varphi(y)$; since $\varphi$ is a product of elementary automorphisms, this gives rise to a change of variables in $k_q(x,y)$, i.e. $k_q(x,y) = k_q(u,v)$.  

We would like to show that $c_z$ acts as an elementary automorphism on $u$ and $v$.  While $z$ is an element of $k_q(y)(\!(x)\!)$ and is not necessarily in $k_q(v)(\!(u)\!)$, $c_z$ is still a well-defined automorphism on $k_q(u,v)$ and so $c_z(v) = zvz^{-1} \in k_q(u,v)$.

Meanwhile $c_z$ fixes $u$, which $q$-commutes with both $v$ and $c_z(v)$; it is easy to see that $u$ must therefore commute with $c_z(v)v^{-1}$.  The centralizer of $u$ in $k_q(u,v)$ is precisely $k(u)$, so $c_z(v)v^{-1} = a(u) \in k(u)$.  Now
\[c_z(u) = u, \quad c_z(v) = c_z(v)v^{-1}v = a(u)v\]
is elementary as required.

Let $a(x) \in k(x)$ be the element obtained by replacing every occurrence of $u$ in $a(u)$ by $x$.  We can define an elementary automorphism on $k_q(x,y)$ by $h: x \mapsto x, \ y \mapsto a(x)y$, which allows us to write the action of $c_z$ as follows:
\[c_z(u) = u = \varphi \circ h(x), \quad c_z(v) = a(u)v = \varphi(a(x)y) = \varphi \circ h (y).\]
Hence $c_z \circ \varphi = \varphi \circ h$, and so $c_z = \varphi \circ h \circ \varphi^{-1}$ is a product of elementary automorphisms as required.  The case $\varphi(y) = r$ follows by a symmetric argument.
\end{proof}
We note that the automorphism $\gamma$ from Proposition~\ref{z_automorphism_order_2} fixes $y-y^{-1} \in E$, but it is not clear whether $y-y^{-1}$ can be written as the image of elementary automorphisms on $E$.

We finish by listing several new open questions raised by these results.
\begin{questions}\ 
\begin{enumerate}
 \item Is there an algorithm to identify elements fixed by a given conjugation map $c_z$ and establish whether they are the image of elementary automorphisms?
\item Does every conjugation automorphism $c_z$ with $z \not\in D$ decompose into a product of elementary automorphisms and an inner automorphism?  In particular, does $\gamma$ from Lemma~\ref{gamma_def} decompose in this fashion?
\item An automorphism of order 5 is defined on $D$ in \cite[\S3.3.2]{Dumas2} by
\[\pi:\ x \mapsto y, \quad y \mapsto x^{-1}(y + q^{-1}).\]
Based on preliminary computations we conjecture that $D^{\pi} \cong D$ as well.  Does $D$ admit any other automorphisms of finite order, and in particular any such automorphisms with fixed rings which are not $q$-division?
\item The ``non-bijective conjugation'' map in Proposition~\ref{z_endomorphism} gives rise to a doubly-infinite chain of $q$-division rings
\[\ldots \subsetneq z^2Dz^{-2} \subsetneq zDz^{-1} \subsetneq D \subsetneq z^{-1}Dz \subsetneq z^{-2}Dz^2 \subsetneq \dots\]
 What can be said about the limits
\[
 \bigcup_{i \geq 0} z^{-i}Dz^i\textrm{ and }\bigcap_{i \geq 0} z^iDz^{-i} \ \ ?
\]
\end{enumerate}
\end{questions}

\end{document}